\RequirePackage[british]{babel}
\documentclass[a4paper,leqno,twoside]{amsart}
\usepackage{amssymb}
\usepackage{amscd}
\usepackage{url}
\usepackage{tabularx}

\theoremstyle{plain}
\newtheorem{thm}{Theorem}
\newtheorem{cor}[thm]{Corollary}
\newtheorem{prop}[thm]{Proposition}
\newtheorem{lem}[thm]{Lemma}

\newtheorem*{mt}{Main Theorem}

\theoremstyle{definition}
\newtheorem{df}{Definition}

\theoremstyle{remark}
\newtheorem{rmk}{Remark}

\newtheorem*{1c}{First case}
\newtheorem*{2c}{Second case}

\newcommand{\ie}{\textit{i.e. }}

\newcommand{\cE}{\mathcal E}
\newcommand{\cL}{\mathcal L}
\newcommand{\cR}{\mathcal R}
\newcommand{\cF}{\mathcal F}

\newcommand{\ZZ}{\mathbb{Z}}
\newcommand{\p}{\mathbb{P}}

\newcommand{\R}{\mathcal{R}}

\newcommand{\NN}{\mathbb{N}}

\newcommand{\Gr}{\mathbb{G}}

\newcommand{\OO}{\mathcal{O}}

\newcommand{\II}{\mathcal{I}}

\newcommand{\abs}[1]{\lvert#1\rvert}

\newcommand{\gen}[1]{\langle#1\rangle}
\newcommand{\Z}{{\mathbb Z}}
\newcommand{\CC}{{\mathbb C}}
\newcommand{\de}{\partial}

\DeclareMathOperator{\cacl}{CaCl}
\DeclareMathOperator{\cod}{Codim}

\DeclareMathOperator{\pic}{Pic}
\DeclareMathOperator{\sym}{Sym}

\DeclareMathOperator{\Ext}{Ext}

\DeclareMathOperator{\im}{Im}

\DeclareMathOperator{\Cl}{Cl}
\def\cocoa{{\hbox{\rm C\kern-.13em o\kern-.07em C\kern-.13em o\kern-.15em A}}}
\DeclareMathOperator{\PGL}{\mathbb{P}GL}

\DeclareMathOperator{\proj}{Proj}

\DeclareMathOperator{\Div}{Div}

\begin{document}

\title{On subcanonical Gorenstein varieties and apolarity}
\author{Pietro De Poi \and Francesco Zucconi}

\thanks{}

\address{ 
Dipartimento di Matematica e Informatica\\
Universit\`a degli Stud\^\i\ di Udine\\
Via delle Scienze, 206\\
Loc. Rizzi\\
33100 Udine\\
Italy
}

\email{pietro.depoi@uniud.it} \email{francesco.zucconi@uniud.it}
\date{\today}

\begin{abstract} 
Let $X$ be a codimension $1$ subvariety of dimension 
$>1$ of a variety of minimal degree $Y$. If $X$ is
subcanonical with Gorenstein canonical 
singularities admitting a crepant resolution, then $X$ is 
Arithmetically Gorenstein and 
we characterise such subvarieties $X$ of $Y$ via apolarity as 
those whose apolar hypersurfaces are Fermat. 

\end{abstract}
\maketitle

\bibliographystyle{amsalpha}

\section{Introduction}
The Kodaira vanishing theorem does not extend to a normal Gorenstein
variety, see \cite[Section 3.3]{GR}. 
On the other hand, 
it is known that schemes with Gorenstein canonical singularities 
have a single 
sheaf, denoted by $\omega_X$, that fits perfectly for both duality 
and vanishing theorems, see \cite[Corollary 11.13]{Ko}. 

Then---even if in the context of Mori theory the Gorenstein assumption 
is too restrictive---to understand normal varieties with Gorenstein 
(strictly) canonical 
singularities is quite important and in this paper we prove some 
theorems on them. 

We recall that a variety is called \emph{$s$-subcanonical} if  
the dualising sheaf $\omega_X$ exists and $\cL^{\otimes s}$ is 
isomorphic
to $\omega_X$ where $s\in \Z$, 
$\cL=\phi^{*}\OO_{\p^N}(1)$, and  $\phi\colon X\rightarrow \p^N$ 
is an embedding. The ring 
$\cR:=\bigoplus_{i=0}^{\infty}H^{0}(X,\cL^{\otimes i})$ is 
known as the \emph{canonical ring} in the case $s=1$ and as the 
\emph{anticanonical ring} in the case $s=-1$. 
A huge amount of study has been devoted to these rings also in a more general setting. 
Here it is sufficient to recall \cite{AS} and \cite{G} only.

In this paper we extensively use the concept of \emph{arithmetically 
Cohen-Macaulay projective variety} (aCM for short), 
see Definition \ref{definizioneuno}, and the concept of \emph{arithmetically 
Gorenstein variety} (aG for short), see Definition \ref{arithmetic}.

We extend the results of \cite{G} with the use of the General Kodaira Vanishing, 
see \cite[Theorem 2.17]{Ko}.

We consider normal projective $n$-dimensional aCM varieties 
$X\subset\p^N$ with canonical Gorenstein singularities that are 
regular (i.e. $h^{1,0}(X)=0$) for which $\omega_X$ is base point 
free, the image of
the canonical map $\varphi_{\omega_X}$ has maximal dimension and 
$h^0(X,\omega_X)\ge n+2$. We show that for such varieties the 
canonical ring $\cR$ of $X$ is generated in degree $n$ unless the image 
of the canonical map is a variety of minimal degree, in which case 
 $\cR$ is generated by elements of degree at most $n+1$. 
The above results are also  
generalised to $s$-subcanonical varieties; see Proposition \ref{lem:green}. 

We also give conditions under which a projective variety 
$X\subset\p^N$ with canonical Gorenstein singularities is aG. We show 
(See Theorem \ref{primo} and Theorem \ref{secondo}) that this happens 
if $X$ is $s$-subcanonical and $\ell$-normal, for all $\ell$ with
$0\le \ell\le n+s-1$ and, if $s\geq 
0$, satisfies the additional condition that
$h^i(X,\OO_X(k))=0$ for $1\le i\le n-1$ and $0\le k \le s$.

%
%

We point out that
this result can be related also to \cite[Proposition 6.9]{Va}. 

In this paper we write that a variety $X$ is $s$-aG if it is aG and 
$s$-subcanonical, and we call \emph{$s$-subcanonically regular} a variety which is $s$-aG with 
Gorenstein canonical singularities. 
By Theorem~\ref{secondo} we can extend the approach of our previous works 
\cite{DZ1} and \cite{DZ2} to this class of varieties. 
We recall that F. Macaulay proved that an Artinian graded Gorenstein ring and so of
socle dimension $1$ and degree---say---$k$ can be
realised as $A=\frac{\CC[\de_0,\dotsc,
\de_m]}{F^\perp}$, where $F\in \CC[x_0,\dotsc, x_m]$ 
is a homogeneous polynomial of degree $k$ and $F^\perp:=
\{ D\in \CC[\de_0,\dotsc, 
\de_m]\vert\ D(F)=0 \}$, where  $\CC[\de_0,\dotsc,
\de_m]$ is the polynomial ring generated by the natural
derivations over $\CC[x_0,\dotsc,
x_m]$, see \cite{Mac} and \cite[Section 2.3]{ik}. We call $F$ the 
\emph{Macaulay polynomial} of $A$. Now let $X\subset \p^N$ be an 
$s$-subcanonically regular variety of 
dimension $n$, with homogeneous coordinate ring $S_X$. A choice of 
$n+1$ independent linear forms $\eta_0,\dotsc,\eta_n$, determines an 
Artinian graded Gorenstein ring $\frac{S_X}{(\eta_0,\dotsc,\eta_n)}$ of
socle 
of degree $s+n+1$. The Macaulay polynomial of 
this ring is a form $F_{\eta_0,\dotsc,\eta_n}$ 
of degree  $s+n+1$ in  $N-n$ variables. This defines a rational map:
\begin{equation}\label{eq:ac} 
\alpha_X\colon\Gr(m,N)\dashrightarrow H_{m,s+n+1}
\end{equation}
from the Grassmannian of $m$-planes in $\p^N$ (where $m=N-(n+1)$) to 
the space of homogeneous polynomials of degree $s+n+1$
in $\check\p^{N}$ modulo the action of $\PGL(m+1,\CC)$, sending the 
$m$-plane $(\eta_0=\ldots=\eta_n=0)$ to the orbit of $F_{\eta_0,\dotsc,\eta_n}$.

In this paper we prove natural generalisations of 
the celebrated Noether and Enriques-Petri-Babbage 
Theorems \cite[\S III.3]{ACGH} in the wider context of $s$-subcanonically regular
varieties:

\begin{mt} 
Let $(X,\cL)$ be a polarised $(k-1)$-dimensional variety,
such that $X\subset\abs{\cL}^\vee=:\check\p^N$ 
is an $s$-subcanonical variety  with crepant resolution and with $k> 2$,  $k+s>2$; then   
$X$ is contained as a codimension one subvariety in a rational normal scroll, or a quadric, or 
a cone on the Veronese surface $v_{2}(\p^{2})$ 
if and only if it is  $s$-subcanonically regular and 
for every $k$-tuple of general sections
$\eta_1,\dotsc, \eta_k\in H^{0}(X,\cL)$,
$F_{\eta_{1},\dotsc, \eta_{k}}\in \mathbb C[x_{0},\dotsc, x_{N-k}]$ 
is a Fermat hypersurface of degree $(s+k)$. 
\end{mt}

See Theorem \ref{miofratellos}. 
Main Theorem is the first step to study the geometry of an $s$-subcanonically 
regular variety of dimension $n$ via the behaviour of the rational 
map $\alpha_X\colon\Gr(m,N)\dashrightarrow H_{m,s+n+1}$. For a 
non-trivial example 
concerning the canonical curve case see \cite{BCN}. 
Moreover, we stress the fact that, contrary to the curve 
case---that is $k=2$---it follows, from Main Theorem, that  
given a variety $X$ of dimension $\geq 2$ with Gorenstein strictly 
canonical singularities, to be contained as
a divisor in a rational normal scroll 
and to be subcanonical forces $X$ to be aCM (and hence aG).

We think that 
the assumption that the resolution is crepant establishes an 
interesting link between the theory of singularities and the theory of apolarity. 
Finally we think 
that some of the geometry we have 
described could shed some light on some aspects of Artinian 
Gorenstein Rings, see \cite{cv}.

\section{Preliminaries}



In this paper we will work with projective varieties and schemes over the complex field $\CC$. For us, 
a \emph{variety} $X$ will always be irreducible but not necessarily smooth. 

\subsection{Generalisation of some results of Green}
Nowadays the results of \cite{G} are easily generalisable to many classes of 
varieties. For this work, we assume that $X$ is normal with Gorenstein 
canonical singularities. We will follow closely the exposition 
of \cite{G} indicating the changes to be made to adapt it to 
our case.

Let $\phi\colon X\rightarrow \p^N$ be a 
morphism and set $\cL=\phi^{*}\OO_{\p^N}(1)$. Since $\cL$ is base point free, 
following \cite[\S 2]{G}, we can form the 
exact sequence:
\begin{equation}\label{valutazione}
0\to Q^*_0\to H^0(X,\cL)\otimes\OO_{X} \to \cL\to 0
\end{equation}
naturally given by the evaluation map $H^0(X,\cL)\otimes\OO_{X}\to \cL$.
Then we see that the natural multiplication 
map $\mu_{d}\colon H^0(X,\cL)\otimes H^0(X,(d-1)\cL)\to H^0(X,d\cL)$ fits 
in the cohomology of the obvious sequence obtained from the sequence 
\eqref{valutazione}: 
\begin{multline*}
0\to H^0(X,Q^*_0\otimes(d-1)\cL) \to H^0(X,\cL)\otimes 
H^0(X,(d-1)\cL)\xrightarrow{\mu_d} H^0(X,d\cL)\to \\
\to H^1(X,Q^*_0\otimes(d-1)\cL) \to H^0(X,\cL)\otimes H^1(X,(d-1)\cL)\to\dotsb
\end{multline*}
Letting $C_d$ be the coker of $\mu_d$ we obtain that: 
\begin{equation*}
C_d\cong\ker(H^1(X,Q^*_0\otimes(d-1)\cL) \to H^0(X,\cL)\otimes H^1(X,(d-1)\cL)). 
\end{equation*} 
We are ready to prove:

\begin{prop}\label{lem:green}
Let $X$ be an $n$-dimensional variety with normal Gorenstein canonical 
singularities. Let $\cL$ be a line bundle on $X$ such that: 
\begin{enumerate}
\item $\cL^{\otimes s}\cong \omega_X$; 
\item $\abs{\cL}$ is base point free; 
\item the map associated to $\abs{\cL}$, $\varphi_{\abs{\cL}}\colon X\to \p^N$ 
is such that $\dim(\varphi_{\abs{\cL}}(X))=n$;
\item $\varphi_{\abs{\cL}}(X)$ is not a variety of minimal degree and has codimension at least two. 
\end{enumerate}
Then the ring $\cR:=\oplus_{d=0}^{+\infty}H^0(X,d\cL)$
is generated by elements of degree at most $n+s-1$ if 
$h^{1}(X,\OO_{X})=0$. 
\end{prop}
\begin{proof} Since $X$ is Cohen-Macaulay, see \cite[Corollary 
   11.13]{Ko}, we can apply Serre Duality, and then 
\begin{equation*}
C_d^*\cong H^{n-1}(X,Q_0\otimes(s-d+1)\cL)/\im( H^0(X,\cL)^*\otimes 
H^{n-1}(X,(s-d+1)\cL)), 
\end{equation*}
The subtle vanishing theorem \cite[Theorem 2.14]{G} holds under our 
hypothesis too. In fact  \cite[Theorem 2.8]{G} holds for a variety 
with normal Gorenstein canonical 
singularities since the general Kodaira vanishing is applicable. Now
\cite[Theorem 1.3]{G} is obviously independent of any 
assumption on the singularities of $X$, while the crucial condition 
in the proof of \cite[Theorem 2.14]{G} is that $X$ is regular. 
Finally to apply our 
generalisation of \cite[Theorem 2.14]{G}, using the same notations as in \cite[Theorem 2.14]{G}, 
we have to put
$p=1$, $k=n-1$, $n=p+k$; so if 
$s-d+1=-k$, that is $d=s+n$, then 
the claim follows verbatim as in the proof of \cite[Theorem 3.9 (3)]{G}
since $h^{n-1}(X,Q_0\otimes(1-n)\cL)=0$ and therefore 
$C_{s+n}=0$.
\end{proof}

\begin{rmk} Notice that under the assumptions of Proposition 
    \ref{lem:green} we have $C_{d}=0$ if $d\geq s+n$ 
    as in the proof of \cite[Theorem 3.9 (3)]{G} since we can apply 
\cite[Theorem 2.8]{G} (2) to obtain $h^{n-1}(X,Q_0\otimes (-n)\cL)=0$ 
and \cite[Theorem 2.8]{G} (1) to obtain $h^{n-1}(X,Q_0\otimes 
(-m)\cL)=0$ where $m>n$.
\end{rmk}

\subsection{Arithmetically Gorenstein schemes}\label{ssec:ags}

Let us fix a closed subscheme $Z$ of $\p^N$ of dimension $n\ge1$ and a
system $x_0,\dotsc,x_N$ of projective coordinates. Let $\II_Z$
be the \emph{sheaf of ideals} of $Z$. The module
$M^r(Z) := \oplus_{t\in\Z} H^r (\p^N,\II_Z(t))$, $1\le r\le n$, 
is called the \emph{$r$-th Hartshorne-Rao module} of $Z$.

\begin{df}\label{definizioneuno}
Let $Z\subset\p^N$ be a a closed subscheme of dimension $n$. Then $Z$ is said to be
\emph{arithmetically Cohen-Macaulay} (\emph{aCM} for short) 
if $N-n$ is equal to the length of a minimal free
resolution of its homogeneous coordinate ring
\begin{equation*}
S_Z:=\frac{S}{I_{Z}}
\end{equation*}
as an $S$-module, where $S:=\CC[x_0,\dotsc,x_N]$ and 
$I_{Z} := \oplus_{t\in\Z} H^0 (\p^N,\II_Z(t))$ is the homogeneous ideal of $Z$. 
\end{df}
A zero-dimensional scheme is automatically aCM, see \cite[page 10]{Mig}. 

\begin{prop}\label{rmk:1}
A scheme $Z\subset\p^{N}$ of dimension $\ge 1$ is aCM if and only if  $M^r (Z)=0$ for all 
$1\leq r\leq n$.
\end{prop}
\begin{proof} See for example \cite[1.2.2 and 1.2.3]{Mig}. 
\end{proof}

\begin{df}\label{arithmetic}
Let $Z\subset\p^N$ be a closed subscheme. $Z$ is said to be \emph{arithmetically Gorenstein}
(\emph{aG} for short) if it is aCM 
and the last free module of a minimal free resolution of $S_Z$ has rank $1$. 
\end{df}

\begin{df}\label{df:can}
A subscheme $Z\subset\p^N$ is said to be \emph{subcanonical} if there exists 
an integer $s\in \Z$ such that $\omega_Z\cong\OO_Z(s)$, 
where $\omega_Z$ is the dualising sheaf of $Z$ (which exists by \cite[Proposition III.7.5]{H}). 
\end{df}

Note that an aCM scheme is aG if and only if it is subcanonical. The dualising 
sheaf of an $n$-dimensional projective scheme $Z\subset \p^{N}$ is
\begin{equation*} 
\omega_Z={{\cE}}xt_{\OO_{\p^{N}}}^{N-n}(\OO_{Z},\omega_{\p^{N}})
\end{equation*}
which is the  sheafification of the canonical model $\Ext_S^{N-n}(S_Z,S)(-N -1)$. 
By Serre's correspondence, if $Z$ is aCM then the canonical model of the 
last free module of the minimal free resolution of 
$S_Z$ has rank $1$ exactly when $\omega_Z\cong\OO_Z(s)$ for some 
$s\in \Z$. More precisely we have the following:
\begin{prop}\label{senzanome}
If $Z$ is an aCM closed subscheme of $\p^N$, then the following are equivalent:
\begin{enumerate}
\item $Z$ is aG;
\item \label{senzanome2} $\omega_Z\cong\OO_Z(s)$ for some integer $s$;
\item \label{senzanome3} the minimal free resolution of $S_Z$ is 
self-dual, up to a twist. 
\end{enumerate}
\end{prop}
\begin{proof} 
See \cite[Proposition 4.1.1]{Mig}.
\end{proof}

\noindent We need the following:
\begin{df}
Let $Z\subset\p^N$ be a projective closed subscheme. We say that $Z$ is \emph{$j$-normal}, 
with $j\in \Z$, $j\ge 0$, if the natural restriction 
map
\begin{equation*}
H^0(\p^{N},\OO_{\p^N}(j))\to H^0(Z,\OO_Z(j))
\end{equation*}
is surjective. We say that $Z$ is \emph{projectively normal}
(\emph{PN} for short) if it is $j$-normal 
$\forall j\in \Z$, $j\ge 0$.  
\end{df}

The following is an easy characterisation of aCM schemes among 
\emph{PN} ones:

\begin{prop}\label{prop:pn} Let $X\subset\p^N$ be an $n$-dimensional 
    PN scheme.
   Then $X$ 
   is aCM if and only if 
\begin{equation*}
h^i(X,\OO_X(j)) =0,\quad\textup{for } 0< i < n \textup{ and } \forall 
j\in\Z.
\end{equation*}
\end{prop}

\begin{proof} By  Proposition \ref{rmk:1}, $X$ 
    is an aCM-scheme iff $h^r(\p^N,\II_X(j))=0$ for every $j\in\Z$ and $1\leq r\leq n$.
    We consider the cohomology of the standard sequence for $X\subset\p^N$:

\begin{equation*}
0\to \II_X(j)\to\OO_{\p^N}(j)\to\OO_X(j)\to 0.
\end{equation*}
\noindent
Since $X$ is PN, then by  Bott's theorem it follows 
$h^{1}(\p^{N},\II_X(j))=0$ where $j\in\mathbb Z$. Using Bott's 
theorem again
it follows that $h^{i+1}(\p^{N},\II_X(j))=
h^{i}(X,\OO_X(j))$ where $0<i<n$ and $j\in\mathbb Z$. 
\end{proof}

Classically a variety $X$ is called \emph{irregular} if $h^1(X,\OO_X)>0$. 
\begin{cor}
let $X\subset\p^N$ be an $n$-dimensional irregular variety. If $n\ge 2$, $X$ cannot be aG. 
\end{cor}

In the case of subcanonical varieties we can get more than what is in 
Proposition~\ref{prop:pn}, 
thanks to the general Kodaira vanishing theorem and some 
generalisations of the results in \cite{G} proved above, see 
Proposition \ref{lem:green}.

\begin{thm}\label{primo}
Let $X\subset\p^N$ be a canonical $n$-dimensional $\ell$-normal 
variety where $1\le \ell \le n$. Assume that $X$ 
has normal Gorenstein canonical singularities and that $h^i(X,\OO_X)=0$ for 
all $1\le i\le n-1$.  
Then $X$ is aG.
\end{thm}

\begin{proof} Since $X$ is canonical then 
   $\OO_X(1)\cong {\omega_X}$. 
   Hence by Proposition \ref{senzanome} 
   we only need to show that $X$ is aCM.

First of all, without loss of generality we can suppose that $X$ has codimension at least two. 

The cohomology of $0\to \II_X(k)\to\OO_{\p^N}(k)\to\OO_X(k)\to 0$
gives:

\begin{equation}\label{solita}
0\to H^0(\p^{N},\II_X(k))\to H^0(\p^{N},\OO_{\p^N}(k))\to 
H^0(X,\OO_X(k))\to H^1(\p^{N}, \II_X(k))\to 0
\end{equation}
and  $H^i(X, \OO_X(k))\cong H^{i+1}(\p^{N},\II_X(k))$ for $0<i<n$ and for every 
$k\in\mathbb Z$. 
By the general Kodaira vanishing theorem,  $h^i(X,\OO_X(k))=0$ for 
$k<0$, $i<n$. Moreover, Serre duality holds and it gives:
\begin{equation*}
h^i(X,\OO_X(k))=h^{n-i}(X,\OO_X(1-k)) 
\end{equation*}
and therefore, since $h^i(X,\OO_X)=0$, we deduce $h^i(X,\OO_X(k))=0$ for 
$1\le i\le n-1$, 
$\forall k \in \Z$.

It remains to show that $X$ is PN. 
%
%
%
By our $\ell$-normality hypothesis, it remains to
prove that $h^1(\p^{N},\II_{X}(k))=0$, if $k> n$. 
Since $h^{1}(X,\OO_{X})=0$, we can apply Proposition \ref{lem:green} 
with $s=1$. Then the canonical ring of $X$ 
is generated in degree $n$. Now we show that $X$ is $(n+1)$-normal,
the case where $k\geq n+2$ works by induction on $k$ 
in the same vein. 
Since $\omega_{X}=\OO_{X}(1)$ and since we have assumed that $X$ 
is $n$-normal we can put the natural homomorphism 
$\sym^{n+1}H^0(X,\omega_X)\rightarrow H^0(X, \omega_X^{\otimes n+1})$ 
and the multiplication map 
$H^0(X,\omega_X)\otimes H^0(X,\omega_X^{\otimes n})\rightarrow 
H^0(X, \omega_X^{\otimes n+1})$ in the following exact commutative diagram:
\begin{equation*}
\begin{CD}
@.0@. 0@. @.\\
@.@AAA@AAA@.\\
@. \sym^{n+1}H^0(X,\omega_X)@>>> H^0(X,\omega_X^{\otimes n+1})@>>> 
H^1(\p^{N},\II_X(n+1))@>>> 0\\
@.@AAA@AAA@.\\
@. H^0(X,\omega_X)\otimes \sym^n 
H^0(X,\omega_X)@>>>H^0(X,\omega_X)\otimes H^0(X,\omega_X^{\otimes n})@>>> 0@. 
\end{CD}
\end{equation*}
obtained by \eqref{solita}.

Let $\xi\in H^1(\p^{N},\II_X(n+1))$. By surjectivity, there exists $\gamma\in 
H^0(X,\omega_X^{\otimes n+1})$ such that $\gamma\mapsto\xi$.
From the surjectivity in the second column, there exists an $\alpha\in 
H^0(X,\omega_X)\otimes \sym^n H^0(X,\omega_X)$ such that 
$\alpha\mapsto\gamma$. Therefore, by the surjectivity in the second row, 
there exists an $\bar\alpha\in H^0(X,\omega_X)\otimes \sym^n 
H^0(X,\omega_X)$ such that $\bar\alpha\mapsto\alpha$. Let 
$\bar\gamma\in \sym^{n+1}H^0(X,\omega_X)$ be the image of 
$\bar\alpha$. By commutativity, $\bar\gamma\mapsto\gamma$. Then 
by the exactness of the first row, $\xi$ must be zero. 
\end{proof}

Because 
of the history of the topic, we like to recall the following corollary 
of Theorem \ref{primo}:

\begin{cor}
Let $X\subset\p^N$ be a smooth canonical regular $2$-normal surface. 
Then $X$ is aCM and therefore is aG.
\end{cor}

The result of Theorem \ref{primo} is also true in the  
$s$-subcanonical case with a mild extra hypothesis; but in this 
case the assumption on singularities is interesting.

\begin{thm}\label{secondo} Let $s\in\Z$.
   Let $X\subset\p^N$ be a $s$-subcanonical $n$-dimensional $\ell$-normal,
   for all $\ell$ such that $1\le \ell\le n+s-1$,
variety with normal Gorenstein canonical singularities. If $s\ge 0$ 
assume also that $h^i(X,\OO_X(k))=0$, for all $i,k\in \ZZ$  such that 
$1\le i\le n-1$ and $0\le k \le s$.  
Then $X$ is aG.
\end{thm}

\begin{proof} The proof is identical to the proof of 
   Theorem \ref{primo} except that in this case 
   $\omega_{X}=\OO_{X}(s)$ and this causes some trivial shifts in 
   the indices.
\end{proof}

We give the following:
\begin{df}\label{subcanonicallyreg}
   Let $s\in\Z$. An $s$-aG variety $X\subset\p^N$  with 
   normal Gorenstein canonical singularities will be called an  
\emph{$s$-subcanonically regular variety}. 
\end{df}

Thanks to Theorem~\ref{secondo}, being an $s$-subcanonically
regular variety is equivalent to requiring that 
 for all $1\le i\le \dim(X)-1$ and for all 
   $0\le k \le s$, $h^i(X,\OO_X(k))=0$ and $h^1(\p^N,\II_X(\ell))=0$,
   for all $1\le \ell\le \dim(X)+s-1$.


\section{On Subcanonically regular varieties}

Theorem~\ref{secondo} makes it possible to extend the results of \cite{DZ1} and 
\cite{DZ2} to some $s$-subcanonically regular varieties; see also the 
introduction.

To do this, we recall a few facts about apolarity theory.

\subsection{Apolarity} An excellent reference for  
apolarity is the notes by A. Geramita, \cite{Ger}.
Let $S:=\CC[x_0,\dotsc, x_N]$ be the polynomial ring in $(N+1)$-variables. 
The algebra of partial derivatives on $S$, 
\begin{equation*}
T:=\CC[\de_0,\dotsc,\de_N],\qquad \de_i:=\frac{\de}{\de_{x_i}},
\end{equation*}
acts on the monomials by:
\begin{equation*}
\de^a \cdot x^b=
\begin{cases}a!\binom{b}{a}x^{b-a} & \text{if $b\ge a$}\\
0 & \text{otherwise}
\end{cases}
\end{equation*}
where $a,b$ are multiindices, $\binom{b}{a}=\prod_i\binom{b_i}{a_i}$, 
$a!:=\prod_i a_{i}!$ etc. 

We can think of $S$ as the algebra of partial derivatives on $T$ 
by defining
\begin{equation*}
x^a \cdot \de^b=
\begin{cases}a!\binom{b}{a}\de^{b-a} & \text{if $b\ge a$}\\
0 & \text{otherwise.}
\end{cases}
\end{equation*}

These actions define a perfect pairing between the 
forms of degree $d$ in $S$ 
and $T$, $\forall d\in\NN$:  
\begin{equation*}
S_d\times T_d\xrightarrow{\cdot} \CC.
\end{equation*}
Indeed, this is nothing but the extension of the duality between vector spaces: 
if $V:=S_1$, then $T_1=V^*$. 

This perfect 
paring shows the natural duality between $\p^N:=\proj(S)$ and 
$\check\p^N=\proj(T)$. 
More precisely, if $(c_0,\dotsc, c_N)\in\check\p^N$, this gives 
$f_c:=\sum_i c_i x_i\in S_1$, and if $D\in T_a$, 
\begin{equation*}
D \cdot f_c^b=
\begin{cases}a!\binom{b}{a}D(c)f_c^{b-a} & \text{if $b\ge a$}\\
0 & \text{otherwise}.
\end{cases}
\end{equation*}
in particular, if $b\ge a$
\begin{equation*}\label{eq:lin}
0=D \cdot f_c^b \iff D(c)=0
\end{equation*}
\noindent where $D(c)$ is the value of the polynomial $D$ at 
$c:=(c_{0},\ldots, c_{N})$.

\begin{df}
We say that two forms, $f\in S$ and $g\in T$ are \emph{apolar} if 
\begin{equation*}
g \cdot f =f\cdot g =0.
\end{equation*}
\end{df}

Let $f\in S_d$ and $F:=V(f)\subset\p^N$ the corresponding hypersurface; 
let us now define 
\begin{equation*}
F^\perp:=\{D\in T\mid D\cdot f=0 \}
\end{equation*}
and
\begin{equation*}
A^F:=\frac{T}{F^\perp}.
\end{equation*}

\begin{lem}
The ring $A^F$ is Artinian Gorenstein of socle 
of 
degree $d$. 
\end{lem}
\begin{proof}
See \cite[\S 2.3 page 67]{ik}.
\end{proof}

\begin{df}
$A^F$ is called the \emph{apolar} Artinian Gorenstein ring of $F$.
\end{df}

The \emph{Macaulay Lemma} asserts that any Artinian Gorenstein ring of socle of degree $d$ 
is apolar to a hypersurface of degree $d$; 
more precisely:  
\begin{lem}
The map 
\begin{equation*}
F\mapsto A^F
\end{equation*}
is a bijection between the hypersurfaces $F\subset\p^N$ of degree 
$d$ and graded Artinian Gorenstein quotient rings 
\begin{equation*}
A:=\frac{T}{I}
\end{equation*}
with socle of degree $d$. 
\end{lem}
\begin{proof}
See \cite[Lemma 2.12 page 67]{ik}.
\end{proof}

If $A=A^F$ then the polynomial $F$ is called the \emph{Macaulay polynomial} 
of $A$.

\subsubsection{Varieties of sum of powers}

Consider a hypersurface $F=V(f)\subset\p^N$ of degree $d$. 
\begin{df}
A subscheme $\Gamma\subset\check\p^N$ is said to be \emph{apolar to $F$} if
\begin{equation*}
I(\Gamma)\subset F^\perp.
\end{equation*}
\end{df}

The \emph{Apolarity Lemma} holds: 

\begin{lem}\label{lem:apo} Let us consider the linear forms
    $\ell_1,\dotsc,\ell_s\in S_1$ and 
let us denote by 
$L_1,\dotsc,L_s\in\check\p^N$ the corresponding points in the dual space and by 
$\Gamma:=\{L_1,\dotsc,L_s\}\subset \check\p^N$ the corresponding zero-dimensional subscheme. 
Then 
\begin{equation*}
\Gamma\ \textup{is apolar to}\ 
F 
= 
V(f),
\iff \exists \lambda_1,\dotsc,\lambda_s \in \CC^*\ \textup{such that}\ f=\lambda_1\ell_1^d+\dotsc+\lambda_s\ell_s^d
\end{equation*}
If $s$ is minimal, then it is called the 
\emph{Waring number} of $F$. 
\end{lem}

\begin{proof}
See \cite[Lemma 1.15 page 12]{ik}.
\end{proof}


\subsection{Macaulay polynomials of $s$-subcanonically regular 
varieties}
In \cite{DZ1}, 
we studied the special case of the canonical curve 
$C\subset\check\p^{g-1}$ of the map introduced in \eqref{eq:ac}.
In fact it is a well-known result that $C$ is aG (see \cite[page 117]{ACGH}).
Therefore, if we take two general linear forms 
$\eta_1,\eta_2\in ({\R_C})_1=H^0(\omega_C)$, then 
$T:= \frac{\R_C}{\gen{\eta_1,\eta_2}}$ 
is Artinian Gorenstein, and its values of the Hilbert function are 
$1, g-2, g-2, 1$. In particular, the socle degree of $T$ is 
$3$, and by the Macaulay Lemma, this defines a hypercubic in 
$\proj(T^*)$. In this way we obtain the rational map 
$\alpha_C\colon\Gr(g-3,g-1)\dashrightarrow H_{g-3,3}$.

\begin{thm}\label{thm:4s}
Let $X\subset\p^{N}$ be an $s$-subcanonically regular variety of 
dimension $n$. 
Let $\eta_0,\dotsc,\eta_n$ be $n+1$ general linear forms on $\p^N$. 
Then the graded $\CC$-algebra 
$A:=\frac{S_X}{\gen{\eta_0,\dotsc,\eta_n}}$ 
is of degree $s+n+1$. 
\end{thm}
\begin{proof}
By Theorem \ref{secondo} and by Proposition~\ref{senzanome}, the 
homogeneous coordinate ring $S_X$ is aG, i.e. it is the homogeneous 
coordinate ring of an aG scheme.

Therefore, also $A=S_X/(\eta_0,\dotsc,\eta_n)$ is aG since 
$\eta_0,\dotsc,\eta_n$ is a regular sequence: see for  example 
\cite[Proposition 3.1.19(b)]{BH}.
The ring $A$ is obviously graded so by Proposition 
\ref{senzanome}\eqref{senzanome3} it has symmetric Hilbert 
function since it is aG. By symmetry, the socle of $A$ is of 
dimension $1$. Now, 
it remains to prove that the socle of $A$ is of degree  $s+n+1$. Let $K_{A}$ be the \emph{canonical model} of $A$, see 
\cite[Definition 3.6.8 page 139, also page 140]{BH} and let 
$a(A)$ be the \emph{$a$-invariant} of $A$, 
see \cite[Definition 3.6.13]{BH}. 
By Proposition \ref{senzanome}\eqref{senzanome2}, $K_X=S_X(s)$, then
$K_A=A(s+n+1)$ since \cite[Corollary 3.6.14]{BH}. In particular, 
$a(A)=s+n+1$ by \cite[Corollary 3.6.14]{BH}.
This means $A_{s+n+1}\neq 0$ and 
$A_i=0$ for  $i\ge s+n+2$ 
(see the remark which follows \cite[Theorem 3.6.19]{BH}). 
\end{proof}

\begin{rmk}
With simple but tedious calculations we could find the values of the Hilbert function of $A$ of the 
preceding theorem. 
\end{rmk}

By Theorem \ref{thm:4s} it easily follows that 
an \emph{$s$-subcanonically regular} variety $X\subset\p^N$ of 
dimension $n$, defines the map $\alpha_X\colon\Gr(m,N)\dashrightarrow 
H_{m,s+n+1}$ presented in the introduction, see \eqref{eq:ac}.

It is natural to extract pieces of information on the geometry of $X$ by 
the nature of this map and, vice versa, to understand some features of 
some Artinian Gorenstein graded $\CC$-algebras. 
We are informed that the realm of Artinian Gorenstein graded 
$\CC$-algebras is huge, but the case where $X$ is a curve is quite 
important, see \cite{cv}. 

In this paper we concentrate mostly on the geometrical aspects of the 
problem. The first step, in the light of \cite{DZ1}, is to 
understand which are the varieties whose Macaulay polynomials are 
Fermat hypersurfaces.


\subsection{Example: the case of hypersurfaces} 

We analyse now a way to obtain subcanonical varieties 
via hypersurfaces of the $M$-dimensional projective space  and Veronese embeddings of $\p^M$. 

So, let us consider a hypersurface $Y\subset\p^M$ of degree $sn+M+1>0$, $s\in \Z$, $n,M\in \NN$; 
therefore, $Y\in\abs{(sn+M+1)H}$, where $H$ is the hyperplane divisor 
on $\p^M$. Assume that $Y$ has only normal Gorenstein canonical singularities.
By adjunction we have:
\begin{equation}\label{eq:2n}
\omega_Y=\OO_{\p^{M}}(snH)\otimes_{\OO_{\p^{M}}}\OO_{Y}.
\end{equation}

Now, consider the $n$-th embedding of $\p^M$, 
$v_n\colon\p^M\to\p^N$ , 
where $N:=\binom{n+M}{n}-1$. $V_n:=v_n(\p^M)$ is a \emph{Veronese 
variety}, and it is well-known that it is 
an aCM variety of degree $n^M$. 

Since $V_n$ is aCM, it follows that
$X:=v_n(Y)\subset V_n$ is aCM also: 

\begin{prop}\label{precedente} 
Let $Y\subset\p^M$ be a hypersurface of degree $sn+M+1$ with normal 
Gorenstein canonical singularities. Let $X$ be its $n$-tuple Veronese embedding 
$X:=v_n(Y)\subset V_n\subset \p^N$, with $N=\binom{n+M}{n}-1$. 
Then $X$ is aG and $s$-subcanonically regular. 
\end{prop}

\begin{proof} 
We need to show that 
$h^i(\p^{N}, \II_X(j))=0$ for all $j\in \Z$, $1\le i \le M-1$. 

By the above inclusions we can
construct the following commutative and exact diagram of sheaves:
\begin{equation*}
\begin{CD}
@.0@.@. @.\\
@.@AAA@. @.\\
@.\II_{X,V_n}(j)@.@. 0@. \\
@.@AAA@. @AAA\\
0 @>>> \II_X(j) @>>> \OO_{\p^N}(j) @> >> \OO_{X}(j) @>>> 0\\ 
@. @AAA @| @AAA\\
0 @>>> \II_{V_n}(j) @>>> \OO_{\p^N}(j) 
@>>> \OO_{V_n}(j) @>>> 0\\ 
@.@AAA@. @AAA\\
@.0@.@. \OO_{V_n}(j-X) \\
@.@.@. @AAA\\
@.@.@. 0
\end{CD}
\end{equation*}
By the above diagram 
$\II_{X,V_n}(j)\cong \OO_{V_n}(j-X)$, and then $h^i(\p^{N},\II_{V_n}(j))=0$
for all $j\in \Z$, $1\le i \le M$ because $V_n$ is aCM. 
Therefore it is sufficient to show that $h^i(\p^{N},\OO_{V_n}(j-X))=0$
for all $j\in \Z$,  $1\le i \le M-1$. 

Since 
$v_n\colon \p^M
\to V_n\subset \p^N$ 
is an embedding, we are
reduced to show only that $h^i(\p^{M},\OO_{\p^M}(jn-sn-M-1))=0$
for all $j\in \Z$, $1\le i \le M-1$, 
which is well-known (see 
\cite[III.5.1]{H}).

Moreover, by 
\eqref{eq:2n}, $X$ is also aG and $s$-subcanonical. 
\end{proof}

\begin{cor}\label{conovero}
Let $C\subset\p^K$ be a cone over the $n$-tuple Veronese embedding 
$V_n\subset \p^N$ of $\p^M$, with $N=\binom{n+M}{n}-1$. Let $Z\subset C$ be a codimension one subvariety with normal 
Gorenstein canonical singularities which is $s$-subcanonical. 
Then $Z$ is $s$-aG and 
it is linearly equivalent to a cone over the image of 
 a hypersurface  $Y\subset\p^M$ of degree $sn+M+1$ under 
 the $n$-tuple Veronese embedding. 
\end{cor}

\begin{proof}
By \cite[Exercise II.6.3]{H}, $\pic(C)\cong\pic(V_n)$, and the isomorphism
is given by the projection from the vertex of the cone. 
Now,  $\pic(V_n)\cong\pic(\p^M)=\ZZ$, so $Z$ is linearly equivalent to a
cone over the image of 
 a hypersurface  $Y\subset\p^M$ of degree $sn+M+1$ under the $n$-tuple Veronese
 embedding. By Proposition \ref{precedente}, 
this cone $D$ is aCM (recall Definition~\ref{definizioneuno} and 
the fact that a cone and its base have isomorphic 
homogeneous coordinate rings) and hence $s$-aG. Moreover, since $X$ is linearly equivalent to $D$, it follows that 
$h^i(\p^{N},\OO_C(j-X))= h^i(\p^{N},\OO_C(j-D))$
for all $j\in \Z$, hence the assertion.  
\end{proof}

Moreover, we deduce immediately:

\begin{prop} Let $Y\subset\p^M$ be a hypersurface of 
   degree  $sn+M+1$ with normal Gorenstein canonical singularities. Let 
$A:=\frac{S_X}{\gen{\eta_0,\dotsc,\eta_{M-1} }}$ be the Artinian 
graded Gorenstein ring corresponding to 
the $s$-subcanonical variety $X:=v_n(Y)\subset\p^N$, 
with $\eta_0,\dotsc,\eta_{M-1}$
general linear forms on $\p^N$. 
Then the Macaulay polynomial of $A$ is an $(s+M)$-tic of Waring number at most $n^M$ 
(\ie it is the sum of at most 
$n^M$ $(s+M)$-th powers of linear forms). 
\end{prop}

\begin{proof}
The estimate for the Waring number is given by the degree of the Veronese variety $V_n=v_n(\p^M)\supset X$. 
\end{proof}

\subsection{Complete intersections}
Perhaps the easiest way to obtain aG varieties is doing complete intersections in 
projective space. In fact, if a variety $X\subset\p^N$ is the 
complete intersection of $c$ hypersurfaces, $F_1\dotsc,F_c$ of 
degrees  $d_1,\dotsc, d_c$, then, by adjunction, we 
have that $\omega_X\cong \OO_{\p^N}(-N-1+\sum_{i=1}^c d_i)\mid_X$.
We set $E:=\oplus_{i=1}^c\OO_{\p^{N}}(-d_i)$. We have the \emph{Koszul 
complex}:
\begin{equation*}
0\to\wedge^c E\to \dotsb \to E \to \II_X\to 0, 
\end{equation*}
and noticing that the intermediate cohomology of $\wedge^i E$ is zero, 
we deduce that $X$ is aCM also. Therefore, $X$ is $s$-aG, with 
$s:=-N-1+\sum_{i=1}^c d_i$. 

\begin{lem}\label{quadrica}
Let $X$ be a codimension one subcanonical subvariety of a quadric 
hypersurface $Q\subset\p^N$. 
Then $X$ is aG.
\end{lem}

\begin{proof}
Let $r$ be the rank of the quadric $Q$; if $r\ge 4$, then, by Klein's theorem, see 
\cite[Exercise II.6.5(d)]{H}, $X$ is a complete intersection, and therefore it is aG. 

If $r=2$, then $\pic(Q)\cong\pic(C)=\ZZ P$, where $C$ is a smooth conic and $P$ is the class of 
a (closed) point of $C$, and the first isomorphism is given by the 
projection of the cone onto its base;   
moreover, the class of the hyperplane section $H$ of $Q$ is double the generator, \ie $H\equiv 2P$ (where 
$\equiv$ is the linear equivalence of divisors) see \cite[Exercises II.6.3 and II.6.5(c)]{H}. In particular, we can think of $Q$ as 
a cone with base $C\subset\p^2$. 
$X$ is linearly equivalent to the subscheme $D$ given by 
$d$ rulings of the cone, where $d:=\deg(X)$, \ie the span of the vertex of $Q$ with the $d$ points $X\cap\p^2$. By adjunction, 
$X\cap\p^2\subset C$ is subcanonical, and hence aG, since it is zero-dimensional. Therefore, $D$ is aG. 
Now, since $X\equiv D$, then $h^i(\p^{N},\OO_Q(j-X))= h^i(\p^{N},\OO_Q(j-D))$
for all $j\in \Z$, hence the assertion.

If $r=3$, then 
$\pic(Q)\cong(Q')\cong \ZZ F_1\oplus \ZZ F_2$,  where $Q'\subset\p^3$ is a smooth quadric and $F_1,F_2$ are the classes of 
the two rulings of  $Q'$, and the first
isomorphism is given by the 
projection of the cone onto its base;   
moreover, the class of the hyperplane section
$H$ of $Q$  is $H\equiv F_1+F_2$ (see again
\cite[Exercises II.6.3 and II.6.5(c)]{H}). 
In particular, we can think of $Q$ as 
a cone with base $Q'\subset\p^3$. $X$ is linearly equivalent
to the cone over a curve $C\subset Q'$; moreover, we can identify $C$ with 
$X\cap Q'\subset\p^3$, and hence $C$ is subcanonical. Suppose
that $C$ is $s$-subcanonical and $[C]=aF_1+bF_2$. Then 
$a=b=s+2$, because on $Q$ every effective divisor moves, so
$C$ is linearly equivalent to a complete
intersection, which 
is aCM, and we can conclude as above that $X$ is aCM and therefore aG. 

\end{proof}


\subsection{$1$-codimensional varieties in rational normal scrolls} 
\label{Ferrario}

For fixing some notations, we recall some basic facts about rational normal scrolls.

\subsubsection{Rational normal scrolls}
By definition, a \emph{rational normal scroll} (RNS for short) 
of type $(a_1,\dotsc,a_k)$, is the image $S_{a_1,\dotsc,a_k}$ of 
the $\p^{k-1}$-bundle 
$\pi\colon\p(\cE)=\p(\OO_{\p^1}(a_1)\oplus\dotsb\oplus
\OO_{\p^1}(a_k))\rightarrow\p^{1}$, under the morphism 
$j\colon\p(\cE)\rightarrow \p^{N}$  defined by the tautological bundle 
$\OO_{\p(\cE)}(1)$. We can arrange the integers $0\leq a_{1}\leq 
a_{2}\leq\ldots\leq a_{k}$ and notice that $N=k-1+f$ where we set $f:=\sum_{i=1}^{k} 
a_i=\deg(S_{a_1,\dotsc,a_k})$. If $a_{1}=a_{2}=\ldots=a_{\ell}=0$ and 
$a_{\ell+1}\neq 0$ where 
$1\leq \ell< k$ then $S_{a_1,\dotsc,a_k}$ is a cone of vertex $V$ of 
dimension $\ell-1$. Notice, in particular, that if 
$a_{1}=a_{2}=\ldots=a_{k-1}=0$ and $a_{k}=1$ then  
$S_{a_1,\dotsc,a_k}=\p^{N}$ where $N=k$. 
Since our theory is for varieties $X$ contained in a 
projective space, in the case of varieties 
$X\subset S_{a_1,\dotsc,a_k}$ we need to study the morphism $j\colon 
Y\rightarrow X$ induced by $j\colon\p(\cE)\rightarrow 
S_{a_1,\dotsc,a_k}\subset\p^{N}$ on the $j$-strict transform $Y$ of 
$X$. We follow the well-written exposition on Weil divisors on 
rational normal scrolls \cite{Fe}.

It is well known and easy to see that $j\colon\p(\cE)\rightarrow 
S_{a_1,\dotsc,a_k}\subset\p^{N}$ is a rational resolution of the 
singularity $V$. In particular the Weil divisors are Cartier divisors 
on $\p(\cE)$ but this is not true in general for 
$S_{a_1,\dotsc,a_k}$. Write $\Cl(X)$ and $\cacl(X)$ for the divisor 
class group and the group of Cartier divisors modulo principal
divisors of $X$, as in \cite{H}. Since $\cod(V,S_{a_1,\dotsc,a_k})\geq 
2$, restriction to the complement of $V$ gives an isomorphism 
$\Cl(S_{a_1,\dotsc,a_k})\rightarrow
\Cl(S_{a_1,\dotsc,a_k}\setminus V)$ by\cite[Proposition II.6.5]{H}. 
Since $j$ is an isomorphism over $j^{-1}(S_{a_1,\dotsc,a_k}\setminus 
V)$, we have a surjective homomorphism $J\colon 
\cacl(\p(\cE))\rightarrow \Cl(S_{a_1,\dotsc,a_k})$ which is an 
isomorphism if $\cod(V,S_{a_1,\dotsc,a_k})> 2$ and has kernel 
isomorphic to $\mathbb Z$, spanned by the class of the exceptional 
divisor $E$ of $j$, if $\cod(V,S_{a_1,\dotsc,a_k})=2$.


In 
\cite{Fe} $J(D)$ is called \emph{the strict image} of a Cartier 
divisor $D\subset\p(\cE)$. In 
particular if $[H]$ and $[F]$ are respectively the class of the tautological 
divisor and the class of the fibre of $\pi$, then 
$\pic(\p(\cE))= [H]\mathbb Z\oplus [F]\mathbb Z$. We denote by $\sim$ 
the numerical equivalence of divisors. 
It easily follows that:
\begin{cor}\label{ferraro} 
    \begin{enumerate}
\item If $\cod(V,S_{a_1,\dotsc,a_k})> 2$, then 
$\Cl(S_{a_1,\dotsc,a_k})=J[H]\ZZ\oplus J[F]\ZZ$; 
\item if $\cod(V,S_{a_1,\dotsc,a_k})= 2$, then $E\sim H-fF$ and 
$\Cl(S_{a_1,\dotsc,a_k})=J[F]\mathbb Z$.
\end{enumerate} 
\end{cor}
 
 We remark that the Weil divisor $J(F)$ on $S_{a_1,\dotsc,a_k}$ is 
 not Cartier if $V\neq\emptyset$ since, as in the standard case of the 
 quadric cone, $J(F)$ is not locally principal in a neighbourhood of 
 $V$. Clearly the theory of Weil divisors on $S_{a_1,\dotsc,a_k}$ 
 splits in two cases according to the codimension of $V$. For 
 what we need, we only stress that given a closed irreducible reduced 
 subscheme $X\subset S_{a_1,\dotsc,a_k}$ of pure codimension $1$ with 
 no embedded components, the proper transform $Y\subset\p(\cE)$ is the 
 closure $\overline{j^{-1}(X\cap((S_{a_1,\dotsc,a_k}\setminus V))}$. 
 Hence, by linearity, the proper transform $Y$ is defined for every Weil 
 divisor $X$ on $S_{a_1,\dotsc,a_k}$, and $J(Y)=X$. 
 Moreover if $\cod(V,S_{a_1,\dotsc,a_k})> 2$, then there exists 
 a unique Cartier divisor $Y$ such that $J(Y)=X$, but 
 if $\cod(V,S_{a_1,\dotsc,a_k})= 2$ then $J(Y+mE)=X$ 
 for every $m\in\mathbb Z$.
 
 Finally to treat the group $\Div(S_{a_1,\dotsc,a_k})$ of
 divisorial sheaves on a singular scroll in terms of 
 $\pic(\p(\cE))$ we recall that 
 $\Div(S_{a_1,\dotsc,a_k})=\frac{\Cl(S_{a_1,\dotsc,a_k})}{\equiv}$
 where $\equiv$ is the \emph{linear equivalence} and, if 
 $\cod(V,S_{a_1,\dotsc ,a_k})> 2$, then $j_{\star}\colon 
 \pic(\p(\cE))\rightarrow \Div(S_{a_1,\dotsc , a_k})$ is an 
 isomorphism. To study also the case 
 $\cod(V,S_{a_1,\dotsc ,a_k})=2$ we recall that letting 
 $\cF^{\vee\vee}$ be the double dual of a sheaf on a \emph{{normal 
 scheme}}, the homomorphism $J$ induces a homomorphism,denoted by 
 the same letter, 
 $J\colon\pic(\p(\cE))\rightarrow \Div(S_{a_1,\dotsc ,a_k})$ 
 given by $\OO_{\p(\cE)}(Y)\mapsto j_{*}(\OO_{\p(\cE)}(Y))^{\vee \vee}$. 
 The homomorphism $J$ is an isomorphism if  $\cod(V,S_{a_1,\ldots,a_k})> 2$. 
 Hence with a minor effort we see that the theory of Cartier divisors 
 and \emph{numerical equivalence} on $\pic(\p(\cE))$ goes 
 parallel to the theory of Weil divisors and \emph{{linear 
 equivalence}} on $\Div(S_{a_1,\dotsc,a_k})$. In particular
 \begin{lem}\label{rita} Let $j\colon\p(\cE)\rightarrow S_{a_1,\dotsc,a_k}
     \subset\p^{N}$ be the desingularisation of a singular scroll of 
     vertex $V$ such that $\cod(V,S_{a_1,\dotsc,a_k})> 2$. Then 
     it holds:
 \begin{enumerate}
\item let $X\equiv aJ(H)+bJ(F)$ be a Weil divisor on 
$S_{a_1,\dotsc,a_k}$; then $Y\sim aH+bF$ for the unique Cartier 
divisor on $\p(\cE)$ such that $J(Y)=X$;
\item $j_{*}\OO_{\p(\cE)}(Y)=\OO_{S_{a_1,\dotsc,a_k}}(X)$;
\item $\OO_{S_{a_1,\dotsc,a_k}}(X)=(\OO_{S_{a_1,\dotsc,a_k}}(X))^{\vee 
\vee}$.
\end{enumerate}     
\end{lem}
\begin{proof} See \cite[Note 3.14 and Corollary 3.20]{Fe}.
\end{proof}

For the special case where $\cod(V,S_{a_1,\dotsc,a_k})=2$ there 
is the problem that given a  Weil divisor $X$ it is not necessarily 
true that $j_{*}\OO_{\p(\cE)}(Y+mE)=\OO_{S_{a_1,\dotsc,a_k}}(X)$ for 
any $m\in\mathbb Z$. To choose a good element in the set 
$J^{-1}[X]=\{[Y+mE]|m\in\mathbb Z\}$, in \cite{Fe}, the 
rational total transform of $X$ is defined as the rational divisor $Y+qE$ where 
$q$ is uniquely obtained by the relation $(Y+qE)\cdot E\cdot 
H^{k-2}=0$. 
Then assuming first that $X$ is effective and that $Y\sim 
aH+bF$, it is easy to see that $q=\frac{b}{f}$ (recall $f=\sum 
a_{i}$) so we can define a unique element $X^{*}=Y+n_{q}E\in 
J^{-1}(X)$ called the \emph{{integral total 
transform}} of $X$ where $n_{q}$ is 
the smallest integer $\geq q$.

\begin{lem}\label{ritabis} Let $j\colon\p(\cE)\rightarrow S_{a_1,\dotsc,a_k}
     \subset\p^{N}$ be the desingularisation of a singular scroll of 
     vertex $V$ such that $\cod(V,S_{a_1,\dotsc,a_k})=2$. Then 
     the integral total transform $X^{*}$ of a Weil divisor 
     $X$ on $S_{a_1,\dotsc,a_k}$ is uniquely determined by the 
     linear equivalence class. Moreover:
 \begin{enumerate}
\item let $X\equiv dJ(F)$ where $d\geq 0$; then 
$X^{*}\sim (m+1)H+(f-h-1)F$ where 
$d-1=mf+h$ $(m\geq -1$ and $0\leq h<f$, $m+1=a+n_{q}$ and 
$f-h-1=fn_{q}-b)$;
\item let $X\equiv dJ(F)$ where $d< 0$; then 
$\OO_{S_{a_1,\dotsc,a_k}}(X)\simeq 
\OO_{S_{a_1,\dotsc,a_k}}(-(m+1)J(H)+(f-h-1)J(F))\simeq
\OO_{S_{a_1,\dotsc,a_k}}(-mJ(H)-(h+1)J(F))$;
\item let $X\equiv dJ(F)$ where $d\in\mathbb Z$; then 
$j_{*}\OO(X^{*})=\OO_{S_{a_1,\dotsc,a_k}}(X)$.
\end{enumerate}     
\end{lem}

Recall that $S_{a_1,\dotsc,a_k}$ is aCM, 
and that in \cite{Sch} it is proved that the dualising sheaf 
$\omega_{S_{a_1,\dotsc,a_k}}$ satisfies:
\begin{equation*}
    \omega_{S_{a_1,\dotsc,a_k}}=j_{*}\OO_{\p(\cE)}(K_{\p(\cE)})Â¥
 \end{equation*}

where the canonical class of $\p(\cE)$ is given by:
 
\begin{equation*}
K_{\p(\cE)}= -kH + (N - k - 1) F. 
\end{equation*}

\subsubsection{Crepant $s$-subcanonical varieties in a RNS}
We consider now $1$-codimensional subvarieties 
contained in rational normal scrolls and we ask under which conditions 
they are $s$-subcanonical. To obtain the classification 
we assume that if $X\subset S_{a_1,\dotsc,a_k}$ 
is a Gorenstein irreducible $s$-subcanonical variety then the 
morphism $j\colon\p(\cE)\rightarrow S_{a_1,\dotsc,a_k}\subset\p^{N}$ induces a 
\emph{crepant} morphism $j_{\mid Y}\colon Y\rightarrow X$ on the 
strict transform $Y$ of $X$. For the notion of crepant morphism see 
\cite{Re}. The crepant condition is a natural one in 
the context of $s$-subcanonically regular varieties.

\begin{lem}\label{lem:pic}
Let $s\in\ZZ$ and let $X\subset S_{a_1,\dotsc,a_k}\subset\p^N$, $k>2$, be 
a Gorenstein irreducible $s$-subcanonical variety. Let $Y$ be the 
$j$-proper transform of $X$. 
Assume that the morphism $j_{\mid Y}\colon Y\rightarrow X$ is 
crepant, that is $j^{*}(\omega_{X})=\omega_{Y}$, and that $Y\sim aH+bF$. 
If $k>-s$ then $s+k=a$ and $b=k+1-N$ or $a=k+s+1$, $b=0$ and $N=k$ (
\ie $X$ is a hypersurface of $S_{a_1,\dotsc,a_k}=\p^N$ where 
$a_{1}=\ldots=a_{k-1}=0$ and $a_{k}=1$). 
\end{lem} 
\begin{proof} The assumption that $X$ is $s$-subcanonical means that 
    $\OO_{X}(s)=\omega_{X}$. The assumption that $X$ is 
    Gorenstein implies that we can use adjunction theory to write 
    $\omega_{X}=\OO_{X}(X)\otimes_{\OO_{X}}\omega_{S_{a_1,\dotsc,a_k}}$. 
    Now we distinguish two cases according to the codimension of the 
    vertex $V$ of the RNS $S_{a_1,\dotsc,a_k}$. To ease reading we 
    remind the reader that $H\in H^{0}(\p({\cE}), 
    \OO_{\p(\cE)}(1))$. We denote by $E$ a general effective divisor of 
    the linear system $\abs{H-\pi^{*}a_{k}F}$.

\begin{1c} $\cod(V,S_{a_1,\dotsc,a_k})> 2$. 
    
\end{1c}    
    
    Let us define a divisor $\eta$ on $\p(\cE)$ as
\begin{equation*}
\eta:=(a-k-s)H+(b+N-k-1)F.
\end{equation*}
First we show that $\eta$ is an effective divisor which does not move.
By definition 
\begin{equation}\label{cinoni}
    \eta-Y=(-k-s)H+(N-k-1)F\sim K_{\p(\cE)}-sH.
    \end{equation} 
    By Lemma 
\ref{rita}, numerical equivalence over $\p(\cE)$ translates to linear 
equivalence over $S_{a_1,\dotsc,a_k}$ and, again by Lemma \ref{rita} 
we have $j_{*}\OO_Y(\eta)=\OO_{X}$ since 
$j_{*}\OO_Y(\eta)=\omega_{X}(-s)$. 
Consider the following exact sequence:
\begin{equation*}
0\to\OO_{\p(\cE)}(\eta-Y)\to \OO_{\p(\cE)}(\eta)\to \OO_Y(\eta)\to 0.
\end{equation*}
Since  $s>-k$ then, by \eqref{cinoni}, $\OO_{\p(\cE)}(\eta-Y)$ has 
no sections, and we have that 
$h^0(\p(\cE),\OO_{\p(\cE)}(\eta-Y))=0$. 

Now it is enough 
to show that $h^0(\p(\cE),\OO_{\p(\cE)}(\eta))=1$.
Since $h^0(Y,\OO_Y(\eta))=h^{0}(X,j_{*}\OO_Y(\eta))=
h^{0}(X,\OO_X)=1$, then it is sufficient to show that 
$h^{1}(\OO_{\p(\cE)}(\eta-Y))=0$. Since $N=k-1+f$ and 
$\pic(\p(\cE))= [H]\mathbb Z\oplus [F]\mathbb Z$ we can write 
$\eta-Y=(-k-s)H -fF$. Thus we have to show that 
$h^1(\p(\cE),\OO_{\p(\cE)}(-k-s)\otimes\OO_{\p(\cE)}((-fF))=0$. 
By \cite[Exercise III.8.4(a)]{H} we deduce that $R^i\pi_*(\OO_{\p(\cE)}(-k-s))=0$ $\forall i>0$ (recall that $-k-s<0$). 
Then by  \cite[Exercise III.8.1]{H}, $H^1(\p(\cE),\OO_{\p(\cE)}(-k-s))\cong H^1(\p^1,\pi_*(\OO_{\p(\cE)}(-k-s)))$, 
but by again  \cite[Exercise III.8.4(a)]{H}, $\pi_*(\OO_{\p(\cE)}(-k-s))=0$. Since $\OO_{\p(\cE)}(F))\cong \pi^*\OO_{\p^1}(1)$, we 
can conclude by the projection formula (\cite[Exercise III.8.3]{H}).

It then follows that $\eta$ is an effective divisor such that 
$h^{0}(\p(\cE), \OO_{\p(\cE)}(\eta))=1$. By 
a trivial numerical computation it follows that these conditions 
imply that there exists an $m\in\mathbb N$ such that $\eta\sim 
mE=m(H-a_{k}F)$ and 
$a_{k}>a_{k-1}$. Moreover, by induction starting from the sequence
\begin{equation*}
0\to\OO_{Y}(\eta-Y)\to \OO_{Y}(\eta)\to \OO_{Y_{|Y}}(\eta)\to 0,
\end{equation*}
we get $\eta\cdot Y^{k-1}=0$.
Then the numerical condition 
$\eta\cdot Y^{k-1}=0$ is equivalent to 
\begin{equation}\label{eaeaea}
ma^{k-2}[a\cdot\sum_{i=1}^{k-1}a_{i} +(k-1)b]=0
\end{equation}
since $H^{k}=\deg S_{ a_1,\dotsc,a_k}=\sum_{i=1}^{k}a_{i}=f$. 
By Equation \eqref{eaeaea}, it follows that either $m=0$, which is the claim, 
or 
\begin{equation}\label{etaetaeta}
b= \frac{a(a_k-f)}{k-1}. 
\end{equation}

If we proceed as we did for $Y$, with $E=H-a_{k}F$, again 
  by induction, starting from the sequence
\begin{equation*}
0\to\OO_E(\eta-Y)\to \OO_E(\eta)\to \OO_{E_{|E}}(\eta)\to 0,
\end{equation*}
we deduce $\eta\cdot E^{k-1}=0$.
From this equation, we infer 
\begin{equation}\label{eaeaea1}
b=a((k-1)a_k -f). 
\end{equation}
If $a\neq 0$, we have, from Equations \eqref{etaetaeta} and \eqref{eaeaea1}, that 
\begin{equation}
(k-2)(a_kk-f)=0; 
\end{equation}
since $k>2$, we have that $f=ka_k$, which cannot happen, since we have supposed that $a_k>a_{k-1}$.

\begin{2c}
$\cod(V,S_{a_1,\dotsc,a_k})=2$.
\end{2c}

We have 
denoted by $E$ the exceptional divisor of 
$j\colon\p(\cE)\rightarrow S_{a_1,\dotsc,a_k}$. The  
argument used above shows that 
$j_{*}\OO_{\p(\cE)}(\eta)=\OO_{S_{a_1,\dotsc,a_k}}$ then, by Corollary 
\ref{ferraro} $(2)$ and by Lemma \ref{ritabis}, we have that 
$(J[\eta])^{*}=0$. In particular there exists an $m\in \mathbb Z$ 
such that $\eta\sim mE$. Then $-s-k+a=m$ and $b+N-k-1=mf$ and
\begin{equation}\label{cinonibis}
     \eta=mE\sim K_{\p(\cE)}+Y-sH.
\end{equation}

In particular since $\OO_{Y}(sH)=j^{*}\OO_{X}(s)$ we have 
$\OO_{Y}(mE)\simeq\omega_{Y}-j^{*}\omega_{X}=0$. Hence $b=0$ or 
$m=0$. 

Assume $m\neq 0$. By Equation \eqref{cinonibis} we have that 
$-s-k+a=m$ and $N-k-1=-mf$. Since $N=f+k-1$ it follows that $(m+1)f=2$ 
and the only solution---since $f\geq 0$ and we 
have assumed $m\neq 0$---is $f=1$ and $m=1$. 
\end{proof}

\begin{rmk}
The case $k=2$ in Lemma \ref{lem:pic} was considered in \cite[Lemma 15]{DZ2}. 

Moreover, if $X$ is smooth, it is well known that $k\ge -s$ and $s=-k$
if and only if our variety is $X$ is a $\p^{k-1}$. The same results 
holds if $X$ is 
normal and Gorenstein: see \cite[Proposition 3.4, Paragraph 4.13 and 
Theorem 5.15]{Fu2}. 
\end{rmk}
We can sum up the above general results in the following:

\begin{prop}\label{prop:4gonas} Let $k\in \mathbb{N}$ be such that $k> 2$,  
and let $s\in\ZZ$ be such that $k>-s$.
Let $X\subset S_{a_1,\dotsc,a_k}$ be a Gorenstein irreducible
$1$-codimensional subvariety 
and let $Y\subset\p(\cE)$ be its $j$-proper transform. Then $X$ is general 
in its linear equivalence class and $s$-subcanonical, 
and $j_{\mid Y}\colon Y\rightarrow 
X$ is crepant, if and only if either $Y$ is a general element of 
$\abs{(s+k)H+(2-f) F}$, or $k=N$ and $Y$ is a hypersurface of degree $(s+k+1)$. 
\end{prop}
\begin{proof} By Lemma \ref{lem:pic} 
    we have only to show that if $Y\in |(s+k)H+(k+1-N) F|$ 
    or $Y\in |(s+k+1)H|$ and $N=k$ is general 
    then $X=j(Y)$ is an $s$-subcanonical Gorenstein irreducible
$1$-codimensional subvariety and  $j_{\mid Y}\colon Y\rightarrow 
X$ is crepant. In the case where $Y\in |(s+k+1)H|$ the morphism 
$S_{a_{1},\ldots, a_{k}}\to\p^{N}$ is the blow-up along the a 
codimension $2$ linear subspace of $\p^{N}$. Hence $Y$ is the total 
transform of a general hypersurface $X$ of degree $(s+k+1)$. 
Now the assumption that $Y$ is general forces that 
$X$ is smooth. If $\cod(V,S_{a_1,\dotsc,a_k})>2$ then the claim 
follows. 
%
If $\cod(V,S_{a_1,\dotsc,a_k})=2$ and $f\geq 2$ then by the 
argument of the second case of the proof of Lemma \ref{lem:pic} we 
have that, taking $Y\in |(s+k)H+(2-f) F|$, it necessarily follows by 
adjunction that 
$j_{Y}\colon Y\rightarrow X=j(Y)$ is crepant.
\end{proof}

\subsubsection{On aG-subvarieties in a RNS}
Next, we show that an $s$-subcanonical, 
$1$-codimensional subvariety 
$X\subset S_{a_1,\dotsc,a_k}$ such that $j_{\mid Y}\colon Y\rightarrow 
X$ is crepant is aG: 

\begin{prop}\label{miazias} Let $k\in \NN$ be such that $k> 2$ and 
let $s\in\ZZ$ be such that $k>-s$.
Let $X\subset S_{a_1,\dotsc,a_k}$ be a Gorenstein irreducible
$1$-codimensional subvariety, $s$-subcanonical, general 
in its linearly equivalence class, and such that $j_{\mid Y}\colon Y\rightarrow 
X$ is crepant, where $Y\subset\p(\cE)$ is its $j$-proper transform. 
Then $X$ is aG.
\end{prop}

\begin{proof} 
Set $S:=S_{ a_1,\dotsc,a_k}\subset\mathbb P^N$. 
By the natural inclusions $X\subset S\subset\p^N$ we can
construct the following exact and commutative diagram of sheaves:
\begin{equation*}
\begin{CD}
@.0@.@. @.\\
@.@AAA@. @.\\
@.\II_{X,S}(m)@.@. 0@. \\
@.@AAA@. @AAA\\
0 @>>> \II_X(m) @>>> \OO_{\p^N}(m) @> >> \OO_{X}(m) @>>> 0\\ 
@. @AAA @| @AAA\\
0 @>>> \II_{S}(m) @>>> \OO_{\p^N}(m) 
@>>> \OO_{S}(m) @>>> 0\\ 
@.@AAA@. @AAA\\
@.0@.@. \OO_{S}(m)\otimes_{\OO_{S}}\OO_{S}(-X) \\
@.@.@. @AAA\\
@.@.@. 0
\end{CD}
\end{equation*}
We need to show that 
$h^i(\p^{N},\II_X(m))=0$ for all $m\in\ZZ $ and $1\le i\le k-1$. 
Set $\OO_{S}(m-X):=\OO_{S}(m)\otimes_{\OO_{S}}\OO_{S}(-X)$.
By the above diagram, 
$\II_{X,S}(m)\cong \OO_{S}(m-X)$ and, since $S$
is aCM, we have $h^i(\p^{N},\II_S(m))=0$ for all $m\in \Z$, $1\le i \le k$; 
therefore it is sufficient to show that $h^i(S,\OO_S(m-X))=0$ for all
$m\in \ZZ$, $1\le i \le k-1$.

\begin{1c}{$\cod(V,S_{a_1,\dotsc,a_k})> 2$ or 
$N>k$. }
\end{1c}

This assumption, with Proposition \ref{prop:4gonas},  
gives $\OO_S(m-X)=j_{*}\OO_{\p(\cE)}(K_{\p(\cE)}+(m-s)H)$. 
It also gives the crucial inequality: $f=N-k+1\geq 2$.

Now, since 
$K_{\p(\cE)}+(m-s)H\sim (m-s-k)H+(f-2)F$, then by the projection 
formula (see \cite[exercise III.8.3]{H}) 
applied to the morphism $j$, we have that 
$R^{i}j_{*}\OO_{\p(\cE)}(K_{\p(\cE)}+(m-s)H)=
R^{i}j_{*}\OO_{\p(\cE)}((f-2)F)\otimes_{\OO_{S}}\OO_{S}(m-s-k)$. 
The variety $W$ contracted by $j$ 
is a rational scroll and, since $f\geq 2$, we have both
$H^{i}(W,\OO_{W}((f-2)F))=0$ and hence 
$R^{i}j_{*}\OO_{\p(\cE)}((f-2)F)=0$, for 
$1\leq i\leq k-1$. By \cite[Exercise III.8.1]{H}, we then have that 
$H^{i}(S, \OO_S(m-X))=H^{i}(\p(\cE),\OO_{\p(\cE)}(K_{\p(\cE)}+(m-s)H)))$, 
where $m\in\Z$ 
and $i\geq 0$. To show the claim we need to show that 
$h^i(\p(\cE),\OO_{\p(\cE)}(K_{\p(\cE)}+\ell H))=0$, where $-s\leq \ell\leq 0$. In 
fact, 
for all $\ell\in \ZZ$, $1\le i \le k-1$ it holds that 
$h^i(\p(\cE),\OO_{\p(\cE)}(K_{\p(\cE)}+\ell H))=0$. 
If $\ell>0$, then 
$h^i(\p(\cE),\OO_{\p(\cE)}(K_{\p(\cE)}+\ell H))=0$, by the
Kodaira vanishing theorem. 
If $\ell\leq 0$, then it is sufficient to prove, by Serre duality, that
$h^i(\p(\cE),\OO_{\p(\cE)}(t H))=0$, with $t\ge 0$, 
$1\le i \le k-1$. 
Let $\pi$ be, as usual, the natural projection map,
$\pi\colon\p(\cE)\to\p^1$; first of all, we have that 
$R^i\pi_*\OO_X(t H)=0$ if $t\ge 0$, 
$1\le i \le k-1$ (see for example \cite[Exercise III.8.4(a)]{H}).
But then $H^i(\p(\cE),\OO_{\p(\cE)}(t H))=
H^i(\p^1,\pi_*\OO_{\p(\cE)}(t H))$ (see for
example \cite[Exercise III.8.1]{H}) and we conclude, since 
$\pi_*\OO_{\p(\cE)}(\ell H)\cong
\sym^{t}(\OO_{\p^1}(a_1)\oplus\dotsm\oplus\OO_{\p^1}(a_k))$ 
(again by \cite[Exercise III.8.4(a)]{H}). We have shown that $X$ is 
aG.

\begin{2c}{ $\cod(V,S_{a_1,\dotsc,a_k})=2$ and 
$f=1$. } 
\end{2c}

In this case $j\colon S_{0,\ldots,0,1}\rightarrow \p^{k}$ and by 
Proposition \ref{prop:4gonas}, $Y\in |(s+1+k)H|$. Hence the claim is 
easy.
\end{proof}

With the hypothesis of Proposition \ref{miazias}, 
we can give a converse to Proposition~\ref{prop:4gonas}: that is
we can characterise  $s$-subcanonically regular varieties with a crepant 
resolution. This is the Main Theorem of the introduction.

\begin{thm}\label{miofratellos}  Let $k\in \NN$ be such that $k> 2$ and  
   let $s\in\ZZ$ be such that $k+s>2$.
Let $(X,\cL)$ be a polarised $(k-1)$-dimensional variety,
such that $X\subset\abs{\cL}^\vee=:\check\p^N$ 
is an $s$-subcanonical variety  with crepant resolution; then   
$X$ is contained as a codimension one subvariety in a rational 
normal $k$-dimensional scroll 
$S_{a_1,\dotsc, a_k}$ or a quadric or 
a cone on the Veronese surface $v_{2}(\p^{2})$ 
if and only if it is  $s$-subcanonically regular and 
for every $k$-tuple of general sections
$\eta_1,\dotsc, \eta_k\in H^{0}(X,\cL)$,
$F_{\eta_{1},\dotsc, \eta_{k}}\in \mathbb C[x_{0},\dotsc, x_{N-k}]$ 
is a Fermat hypersurface of degree $(s+k)$. 
\end{thm}

\begin{proof} 
Since for $N=k$ there is nothing to prove, we can suppose $N>k$.

If $X$ is contained as a divisor in a variety $S$ of minimal degree, 
then by Lemma \ref{quadrica}, by Corollary \ref{conovero} 
and by Proposition \ref{miazias}, 
$X$ is $s$-subcanonically regular. Take $k$ general sections $\eta_1,\dotsc \eta_k\in H^{0}(X,\cL)$ and 
consider $\Gamma:=S \cap V(\eta_1,\dotsc, 
\eta_k)\subset\check\p^{N-k}$. This is a zero-dimensional scheme of length 
$\deg S$ and, because $S$ is aCM, the (saturated) homogeneous ideal $I(\Gamma)$ 
of $\Gamma$ is obtained by adding the linear forms $\eta_1,\dotsc, 
\eta_k$ to $I(S)$, so $I(\Gamma)=(I(S), \eta_1,\dotsc, \eta_k)$. 

%
Then by the Apolarity Lemma~\ref{lem:apo},  $\Gamma$ 
is apolar to 
a Fermat  hypersurface 
$F_{\eta_1,\dotsc, \eta_k}\in \mathbb C[x_{0},\dotsc, x_{N-k}]$ 
(of degree $s+k$, by Theorem~\ref{thm:4s}), 
since $S$ is of minimal degree $(N-k+1)$ which is also the number
of independent linear forms of $\mathbb C[x_{0},\dotsc, x_{N-k}]$.

Conversely, 
let us suppose that our $s$-subcanonical variety $X$ is such that for every
$k$-tuple of general sections
$\eta_1,\dotsc, \eta_k\in H^{0}(X,\cL)$, there is a zero-dimensional scheme of length $\deg S$, 
$\Gamma_{\eta_1,\dotsc,\eta_k}\subset \p^{N-k}:=
V(\eta_1,\dotsc, \eta_k)$ with 
$I(\Gamma_{\eta_1,\dotsc,\eta_k})\subset I(X,\eta_1,\dotsc,\eta_k)$, 
or, in other words, by the 
Apolarity Lemma~\ref{lem:apo}, $\Gamma_{\eta_1,\dotsc,\eta_k}$
is apolar to a Fermat $(s+k)$-tic 
$F_{\eta_1,\dotsc,\eta_k}\in \mathbb C[x_0,\dotsc,
x_{N-k}]$. 

First of all, up to changing coordinates, we can assume that $\eta_1=\de_{N-k+1}, \dotsc, \eta_k=\de_N$. 
Then $\frac{\CC[\de_0,\dotsc,\de_N]}{(I(X),\de_{N-k+1}, \dotsc,\de_N)}$, is Artinian Gorenstein, and can be 
considered as a quotient of the polynomial ring $\CC[\de_0,\dotsc,\de_{N-k}]$. 
Up to a change of coordinates, we can suppose that  
$F_{\de_{N-k+1},\dotsc,\de_N}:=x_0^{s+k}+\dotsb+x_{N-k}^{s+k}$, and we can think of it as a polynomial in  
$\CC[x_0,\dotsc,x_{N-k}]$. 
We now find 
$F_{\de_{N-k+1},\dotsc,\de_N}^\perp$ as an ideal in  $\CC[\de_0,\dotsc,\de_{N-k}]$. 
It is easy to see that 
\begin{equation}\label{eq:riducibs}
F_{\de_{N-k+1},\dotsc,\de_N}^\perp=(\de_i\de_j, \de_i^{s+k}-\de_j^{s+k}),
\quad i,j\in\{0,\dotsc,N-k\},\quad i\neq j, 
\end{equation}
since, by hypothesis, $s+k>2$. 
Then the quadrics of $I(X)$ are of the form 
\begin{equation*}
Q_{i,j}:=\de_i\de_j+\de_{N-k+1}L_{i,j}^1+\dotsb+\de_{N}L_{i,j}^k, 
\end{equation*}
where the $L_{i,j}^\ell$'s are linear forms on $\check\p^N$. 
In 
particular, the vector space of the quadrics vanishing on $X$ has 
dimension $\binom{N-k+1}{2}$; that is: 
$h^0(\p^{N},\II_X(2))= \binom{N-k+1}{2}$.
Since $X$ is PN it follows $(k+1)N-\binom{k}{2}+1 =h^0(X,\OO_X(2))$. 
Let $C$ be a general curve section of $X$. Hence  $3N-3k+6=h^0(C,\OO_C(2))$ 
and 
we can proceed as in Castelnuovo's analysis of curves of maximal 
genus often called 
\emph{Castelnuovo curves}: see 
for example \cite[pages 527--533]{GH}. In fact, 
since $C$ is non-degenerate, no quadric containing $C$ can 
contain a hyperplane $H\cong\check\p^{N+1-k}$, and therefore, 
if we set $\gamma:=C\cap H$, the natural restriction map 
$H^0(\p^{N+2-k},\II_C(2))\to H^0(H,\II_\gamma(2))$ is an 
isomorphism (recall that $C$ is also linearly normal). 

From this, we infer that the points of $\gamma$ impose only 
$h^0(H,\OO_H(2))-h^0(H,\II_\gamma(2))= 2(N+2-k)-1$ conditions 
on quadrics. By Clifford's Theorem, $h^0(C,\OO_C(2))-1\le\deg(C)$, 
therefore $\deg(C)\ge 3(N+2-k)-1>2(N+2-k)+1$ 
since $N>k$. By Castelnuovo's Lemma (see for example 
\cite[page 120]{ACGH}), if $H$ is generic, it follows that $\gamma$ is 
contained in a unique rational normal curve $D$. 

Since $\gamma$ consists of more than $2\deg(D)=2(N+2-k)-2$ points, a 
quadric contains $\gamma$ if and only if it contains $D$. 
Finally, we recall that a rational normal curve in $\p^{N +1-k}$ 
is the intersection of $\binom{N+1-k}{2}$ quadrics, and therefore the intersection 
of the quadrics containing $X$ meets $H\cong\p^{N+1-k}$ exactly in $D$. 
Thus, the intersection of the quadrics containing $X$ is an 
irreducible $k$-dimensional variety $Y$ of minimal degree.
By Bertini's classification theorem of the varieties of minimal degree, see \cite{ei}, it follows that $Y$ is 
a quadric or a RNS or a $k$-dimensional cone on the Veronese surface 
$v_{2}(\p^{2})$.
\end{proof}

\begin{rmk}\label{superficiedidel pezzo} The hypothesis of Theorem \ref{miofratellos}
concerning $s$ cannot be weakened; in fact, 
consider the del Pezzo surface $(Y,\cL)=(\p^{2},\OO_{\p}(3))$. Let $X:=j(Y)$ 
    where $j:=\phi_{|H|}$. 
$X\subset\p^9$ is $(-1)-aG$ but since $X$ 
does not contain plane curves, it cannot be 
contained in a rational normal threefold of $\p^9$. 
\end{rmk}

\begin{rmk}
As we noted in the introduction, the hypothesis $k>2$ is crucial, in the sense that if $k=2$ there are subcanonical curves contained in 
rational normal scrolls which are not PN. For example, 
take a smooth curve $C\in\abs{5(C_0+f)}$ on the rational
normal scroll $S_{1,2}\subset\p^4$, where 
$C_0^2=-1$ is the section at 
infinity and $f$ is the fibre. Then $C$ is $1$-subcanonical of genus $6$ and degree $10$: it is a projection of a 
canonical curve of $\p^5$, and therefore $C$ is not linearly normal. 
\end{rmk}
 
\begin{cor}
If $X$ is a $(k-1)$-dimensional $s$-subcanonical variety with crepant resolution contained as a divisor in a variety of minimal degree $S$ 
such that $s+k>2$, 
such that $X$ is not aCM, then $k=2$, \ie $X$ is a curve and $S$ is a rational normal scroll $S_{a_1,a_2}$ with $a_1\neq a_2$. 
\end{cor}
\begin{proof}
This follows from Lemma \ref{quadrica}, Corollary \ref{conovero}, Proposition \ref{miazias} and Theorem \ref{miofratellos}.  
\end{proof}
\subsection{The case $s+k=2$}
We recall that a  \emph{del Pezzo variety} 
is a pair $(X,H)$, where $X$ is a projective $n$-dimensional 
variety $X$ with only Gorenstein (not necessarily normal) 
singularities and $H$ is an ample Cartier divisor on it 
such that $-K_X=(n-1)H$, and $h^i(X,\OO_X(jH))=0$, for 
all $i,j\in\ZZ$ with $0\le i\le n$. 

For example, by the Kawamata-Viehweg theorem, 
if $X$ is a Gorenstein Fano variety of index $n-1$ with at most canonical 
singularities, it is a del Pezzo variety. 

For a normal Gorenstein del Pezzo variety $X$ of codimension 
$e\geq 2$ it clearly holds that $s+k=2$, but our
Theorem \ref{miofratellos} 
fails in this case for trivial reasons since in Equation \ref{eq:riducibs} the terms $\de_i^{s+k}-\de_j^{s+k}$ $i\neq 
j$ and $i,j=0,\ldots, N-k$ are 
quadratic. Following the argument of the proof we can deduce that the 
homogeneous ideal of 
these projective varieties is generated by quadrics.

Finally we remark that if $X$ is a del Pezzo variety contained in a 
RNS, using the notation of 
Subsection \ref{Ferrario}, it easily follows that either $X\in \abs{2H+(2-f)F}$ 
or $X$ is a hypercubic in the trivial case $k=N$, $S_{a_1,\dotsc,a_k}=\p^N$. 
In the first case, $j^{-1}(X)\in |3H-H_{0}|$ and this is a case of 
Fujita's classification see \cite[Theorem 9.17 page 82]{Fu2} of the del 
Pezzo varieties.

\providecommand{\bysame}{\leavevmode\hbox to3em{\hrulefill}\thinspace}
\providecommand{\MR}{\relax\ifhmode\unskip\space\fi MR }
\providecommand{\MRhref}[2]{%
\href{http://www.ams.org/mathscinet-getitem?mr=#1}{#2}
}
\providecommand{\href}[2]{#2}

\end{document}